\documentclass[12pt]{article}
\RequirePackage[colorlinks,citecolor=blue,urlcolor=blue,linkcolor=blue]{hyperref}
\hypersetup{
colorlinks = true,
citecolor=blue,
urlcolor=blue,
linkcolor=blue,
pdfauthor = {Takahiro Hasebe, Alexey Kuznetsov},
pdfkeywords = {free stable distribution, stable distribution, Mellin transform},
pdftitle = {On free stable distributions},
pdfpagemode = UseNone
}
\usepackage{graphicx,xspace,colortbl}
\usepackage{amsmath,amsthm,amsfonts,amssymb}
\usepackage{color}
\usepackage{enumerate}
\usepackage{fancybox}
\usepackage{epsfig}
\usepackage{subfig}
\usepackage{pdfsync}
    \def\qed{\hfill$\sqcap\kern-8.0pt\hbox{$\sqcup$}$\\}
    \def\beq{\begin{eqnarray}}
    \def\eeq{\end{eqnarray}}
    \def\beqq{\begin{eqnarray*}}
    \def\eeqq{\end{eqnarray*}}

\DeclareMathOperator{\re}{Re}
\DeclareMathOperator{\im}{Im}
    \def\p{{\mathbb P}}
    \def\e{{\mathbb E}}
    \def\r{{\mathbb R}}
    \def\C{{\mathbb C}} 
    	
    \def\d{{\textnormal d}}
    \def\i{{\textnormal i}}

\newtheorem{theorem}{Theorem}
\newtheorem{lemma}{Lemma}

\newtheorem{corollary}{Corollary}
\theoremstyle{definition}

\newtheorem{remark}{Remark}

\title{On free stable distributions}

\author{
{
T. Hasebe
\footnote{Department of Mathematics,  Hokkaido University, Kita 10, Nishi 8, Kitaku, Sapporo 060-0810, Japan. Email: thasebe@math.sci.hokudai.ac.jp}}
,\, 
{
A. Kuznetsov
\footnote{Department of Mathematics and Statistics, 
York University, 
4700 Keele Street, 
Toronto, Ontario, 
M3J 1P3, Canada. Email: akuznets@yorku.ca}}
}
\date{\footnotesize This version: August 18, 2014}

\begin{document}

\maketitle

\begin{abstract}
\bigskip
We investigate analytical properties of free stable distributions and discover many connections with their  classical counterparts. Our main result is an explicit formula for the Mellin transform, which leads to explicit series representations for the characteristic function and for the density of a free stable distribution. All of these formulas bear close resemblance to the corresponding expressions for classical stable distributions. As further applications of our results, we give an alternative 
proof of the  duality law due to Biane and a new factorization of a classical stable random variable into an independent 
(in the classical sense) product of a free stable random variable and a power of a Gamma(2) random variable.
\end{abstract}

{\vskip 0.5cm}
{\vskip 0.5cm}

\section{Introduction and main results}\label{section_intro}

The class ${\mathfrak{W}}^*$ of classical (strictly) stable distributions  is characterized by the following: $\mu \in {\mathfrak{W}}^*$ if and only if for any $c_1>0$ and $c_2>0$ there exists $c_3>0$ such that
\begin{equation}\label{def_stable}
D_{c_1} \mu * D_{c_2} \mu = D_{c_3} \mu,
\end{equation}
where $D_c \mu$ denotes the dilation by $c$ (in other words, $D_c \mu(B)=\mu(c^{-1} B)$ for all Borel sets $B$ in $\r$), and the binary operator ``$*$'' denotes the classical convolution 
\cite{Zolotarev1986}. 
Similarly, the class ${\mathfrak{W}}^{\boxplus}$ of free (strictly) stable distributions   is characterized by the scaling property \eqref{def_stable}, where the classical convolution ``$*$'' is replaced by the free convolution ``$\boxplus$'', see \cite{BPB1999,Pata_1995}. The similarities end at this  stage, and in all other respects these two families of distributions seem to be quite different. For example, the main tool for working with a classical stable distribution is the characteristic function, whereas the free stable distributions are described by their Cauchy transforms. Another difference is that the classical stable distributions have explicit formulas for the Mellin transform and their densities have explicit series expansions (see Sections 2.4, 2.5 and Theorem 2.6.3 in \cite{Zolotarev1986}), whereas the density of a free stable distribution enjoys a representation as an inverse of a rather simple explicit function (see Propositions A.1.2-A.1.4 in \cite{BPB1999}). The latter result does not have an analogue in the case of the classical stable distributions. 

At the same time, several results in the recent papers by Demni \cite{D2011}  and Haagerup and M\"oller \cite{HM2013} have offered 
glimpses of possible deeper similarities and connections between these two families of stable distributions. Demni \cite{D2011} has investigated the appearance of the Kanter random variable both in the classical and free stable distributions, while the Mellin transform of the positive free stable distributions, which was computed by Haagerup and M\"oller in \cite{HM2013}, bears close resemblance to the Mellin transform of the classical stable distributions, see \cite[Theorem 2.6.3]{Zolotarev1986}.

Our main goal in this paper is to investigate analytical properties of the free stable distributions, and to demonstrate close connections with their  classical counterparts. First, let us introduce several notations and definitions and present a new parameterization of the free stable distributions. 

It is known that (up to a scaling parameter) any strictly stable distribution can be uniquely characterized by a pair of parameters 
$(\alpha,\rho)$, which belongs to the {\it set of admissible parameters}
\begin{equation}\label{def_set_A}
{\mathcal A}:=\{\alpha \in (0,1), \; \rho \in [0,1]\} \cup 
\{\alpha\in(1,2], \; \rho \in [1-\alpha^{-1}, \alpha^{-1}]\}. 
\end{equation}
The exceptional case $\alpha=1$ corresponds to a Cauchy distribution, and from now on we exclude this case from consideration. 
The characteristic function of a distribution $\mu_{\alpha,\rho} \in {\mathfrak{W}}^*$ is given by
\begin{equation}\label{def_g_alpha_rho}
{\mathfrak{g}}_{\alpha,\rho}(z):=\int_{\r} e^{\i zx} \mu_{\alpha,\rho}(\d x)=
\exp\left(- z^{\alpha} e^{\pi \i \alpha (1/2-\rho)} \right), \;\;\; z>0. 
\end{equation}

Similarly, a free stable distribution can be parameterized by the pair $(\alpha,\tilde \rho)$, belonging to the set
\begin{equation}\label{def_set_admissible_old}
{\mathcal B}:=\{\alpha \in (0,1)\cup (1,2], \; \tilde\rho \in [0,1]\}.
\end{equation}
A free stable distribution $\nu_{\alpha,\tilde \rho} \in {\mathfrak{W}}^{\boxplus}$ is characterized by the {\it Voiculescu transform}
\begin{equation}\label{def_phi_Voiculescu}
\phi_{\alpha,\tilde \rho}(z)=
\begin{cases}
-e^{\pi \i \alpha \tilde \rho } z^{-\alpha+1}, \; & \textnormal{ if } \; \alpha \in (0,1), \\
e^{\pi \i (\alpha-2) \tilde \rho} z^{-\alpha+1}, \; & \textnormal{ if } \; \alpha \in (1,2],
\end{cases}
\end{equation}
where $\im(z)>0$.  In the above formula and everywhere else in this paper we use the principal branch of the logarithm  and of the power function. As in the classical case, we do not consider the free stable laws with $\alpha=1$, as these correspond to the Cauchy distributions. 
The Voiculescu transform defines the measure $\nu_{\alpha,\tilde \rho}$ in the following way: 
The inverse function of the map $z\mapsto 1/z+\phi_{\alpha,\tilde\rho}(1/z)$ 
is $G_{\alpha,\rho}(z)$, which is the Cauchy transform 
\begin{equation}\label{def_G_alpha_rho}
G_{\alpha,\tilde \rho}(z):=\int_{\r} \frac{\nu_{\alpha,\tilde \rho}(\d x)}{z-x}, \;\;\; \im(z)>0. 
\end{equation}

The parameterization $(\alpha,\tilde \rho)$ which is used to define the Voiculescu transform via 
\eqref{def_phi_Voiculescu} goes back to the seminal work of Bercovici, Pata and Biane \cite{BPB1999} (it also appears implicitly in \cite[Theorem 7.5]{BV1993}). 
In order to present our formulas in a more compact way and to highlight close connections with the classical stable distributions, we need to introduce a new parameterization for the class of
free stable distributions. Instead of using the parameters $(\alpha,\tilde \rho) \in {\mathcal B}$, we will use the pair
\begin{eqnarray}\label{new_parameterization}
(\alpha,\rho)=
\begin{cases}
(\alpha,\tilde \rho), \;\;\; & \textnormal{ if } \; \alpha \in (0,1), \; \tilde \rho \in [0,1], \\
\left(\alpha,\left(1-(2-\alpha)\tilde \rho\right)/{\alpha} \right), 
\;\;\; & \textnormal{ if } \; \alpha \in (1,2], \; \tilde \rho \in [0,1].
\end{cases}
\end{eqnarray}
Note that the map $(\alpha,\tilde\rho) \mapsto (\alpha,\rho)$ is a bijection between the set 
${\mathcal{B}}$ and the set of admissible parameters ${\mathcal A}$.
The benefit of using this new parameterization $(\alpha,\rho) \in {\mathcal A}$ is that the Voiculescu transform of a distribution $\nu_{\alpha,\rho} \in {\mathfrak{W}}^{\boxplus}$ is given by a single expression
\begin{equation}\label{def_phi_Voiculescu_new}
\phi_{\alpha,\rho}(z)=-e^{\pi \i \alpha \rho} z^{-\alpha+1}, 
\end{equation}
whereas the original parameterization \eqref{def_phi_Voiculescu} requires two different formulas, depending on whether $\alpha \in (0,1)$ or $\alpha \in (1,2]$. 
Another advantage of our new parameterization is that the map $\mu_{\alpha,\rho} \in {\mathfrak{W}}^* \mapsto 
\nu_{\alpha,\rho} \in {\mathfrak{W}}^{\boxplus}$ is just the Bercovici-Pata bijection \cite{BPB1999} (up to a scaling parameter); this fact  
can be easily established using Theorem 4.1 in \cite{BN_TH}.

For $(\alpha,\rho) \in {\mathcal A}$ we denote by $X_{\alpha,\rho}$ (respectively, $Y_{\alpha,\rho}$) an $\r$-valued random variable having distribution $\nu_{\alpha,\rho} \in {\mathfrak{W}}^{\boxplus}$ (respectively, $\mu_{\alpha,\rho} \in {\mathfrak{W}}^*$).

Now we are ready to present our results. Their proofs are provided in the next section. 

 \begin{theorem}\label{thm_Mellin}
Assume that $(\alpha,\rho) \in {\mathcal A}$. Then for $s\in (-1,\alpha)$
\begin{equation}\label{eq_Mellin_1} 
\e\left[(X_{\alpha,\rho})^s {\bf 1}_{\{X_{\alpha,\rho}>0\}}\right]=\frac{1}{\pi}\sin(\pi \rho s)\frac{\Gamma(s)\Gamma\left(1-s/{\alpha}\right)}{\Gamma\left(2+s-s/\alpha\right)}. 
\end{equation} 
\end{theorem} 

\vspace{0.25cm}
\noindent
The above expression for the Mellin transform of $X_{\alpha,\rho}$   is remarkably similar to the corresponding result for the classical stable distributions  \cite[Theorem 2.6.3]{Zolotarev1986}:
\begin{equation}\label{eq_Mellin_1_classical} 
\e\left[(Y_{\alpha,\rho})^s {\bf 1}_{\{Y_{\alpha,\rho}>0\}}\right]=\frac{1}{\pi}\sin(\pi \rho s)\Gamma(s)\Gamma\left(1-{s}/{\alpha}\right),  \;\;\; \textnormal{ for } \; -1<s<\alpha. 
\end{equation}

\begin{remark}
Using formula \eqref{eq_Mellin_1} it is easy to prove that, for any $\rho \in [0,1]$,
the random variable $\vert X_{\alpha,\rho} \vert^{-\alpha}$ converges weakly as $\alpha \to 0^+$ to a uniform random variable on the interval $(0,1)$. 
This limiting behavior is different than in the case of the classical stable distributions, where $\vert Y_{\alpha,\rho} 
\vert^{-\alpha}$ converge to an exponential random variable as $\alpha \to 0^+$ \cite{Cressie}. 
The limiting distribution of $(X_{\alpha,1})^{\alpha}$ as $\alpha \to 0^+$ was investigated in \cite{D2011}. However,
the expression given in \cite[Proposition 3]{D2011} is incorrect, as the right-hand side of that formula should be $1/x^2$, instead of $1/x$. 
\end{remark}

\noindent
The following corollary follows easily from Theorem \ref{thm_Mellin}, by taking the limit $s \to 0$ in \eqref{eq_Mellin_1}.
\begin{corollary}\label{corollary_positivity}
Assume that $(\alpha,\rho) \in {\mathcal A}$. Then $\p(X_{\alpha,\rho}>0)=\rho$. 
\end{corollary}

The above result clearly demonstrates the benefit of our new parameterization \eqref{new_parameterization} for free stable distributions. The parameter $\rho$ (unlike $\tilde\rho$) has a very natural interpretation as the positivity parameter, which is consistent with its definition for the classical stable random distributions. 

Theorem \ref{thm_Mellin} can be used to show that a free stable distribution $\nu_{\alpha,\rho}$ is absolutely continuous, with a smooth density $\psi_{\alpha,\rho}(x)$ (this fact was first established in \cite{BPB1999}). The following duality result was obtained by Biane in \cite{BPB1999}, and in this paper we give a new derivation of this result as a simple corollary of Theorem \ref{thm_Mellin}. We would like to emphasize that the same duality relation also holds for densities of classical stable distributions.

 \begin{corollary}\label{corollary_duality}
Assume that $\alpha \ge {1}/{2}$ and $(\alpha,\rho) \in {\mathcal A}$. Then for $x>0$
 \begin{equation}\label{duality_identity1}
 \psi_{\alpha,\rho}(x)=x^{-\alpha-1} \psi_{{1}/{\alpha},\,\alpha\rho}(x^{-\alpha}).
 \end{equation}
\end{corollary} 

\vspace{0.25cm}

We can also express the above duality law in terms of random variables. If $\xi$ is a random variable such that $\p(\xi>0)>0$, we denote by $\hat \xi$ the {\it cutoff} of $\xi$, which is a positive random variable, whose distribution is given by
$$
\p(\hat \xi \in A)=\p(\xi \in A \; | \; \xi >0),
$$
for all Borel sets $A \subseteq (0,\infty)$. 
 It is easy to see that the identity \eqref{duality_identity1} is equivalent to  
 \begin{equation}\label{duality_identity2}
\hat X_{\alpha,\rho}\stackrel{d}{=} \left(\hat X_{{1}/{\alpha},\,\alpha\rho}\right)^{-{1}/{\alpha}}. 
\end{equation}

Another corollary of Theorem \ref{thm_Mellin} is the following distributional identity between stable and free stable laws.

\begin{corollary}\label{corollary1}
Let $Z$ be a Gamma(2) random variable (that is, a positive random variable having the density $p_Z(x)=xe^{-x}$, for $x>0$).  
For $(\alpha,\rho) \in {\mathcal A}$ we have
\begin{equation}\label{idenity_distribution_01}
Y_{\alpha,\rho}\stackrel{d}{=}X_{\alpha,\rho} \times Z^{1-{1}/{\alpha}}, 
\end{equation}
where the random variables $X_{\alpha,\rho}$ and $Z$ are assumed to be independent. 
\end{corollary}

We recall that  {\it the Bercovici-Pata bijection} is an equivalence between the set of classical infinitely divisible
laws and the set of free infinitely divisible laws, which identifies the corresponding generating triplets in the free and classical L\'evy-Khintchine representations \cite{BN_TH,BPB1999}. Our factorization \eqref{idenity_distribution_01} provides a remarkably simple realization of this 
bijection in the case of stable distributions. Furthermore, this result leads naturally to the question of whether there exist other classes of infinitely divisible laws, for which the Bercovici-Pata bijection can be expressed as a similar factorization. 
 
The factorization identity \eqref{idenity_distribution_01} is also related to the following recent result by 
Arizmendi and P\'erez-Abreu \cite{AP2010}: 
\begin{equation}\label{AP}
Y \stackrel{d}{=} S_\theta\times \sqrt{Z_{\theta+2}}, 
\end{equation}
where $Y$, $S_\theta$ and $Z_{\theta+2}$ are independent random variables having the standard normal distribution;  the power semicircle distribution with the density $a_{\theta}(1-x^2)^{\theta+1/2}$, $x\in \r$; and the gamma distribution with the density $b_{\theta}x^{\theta+1}e^{-x/2}$, $x>0$, respectively. It is easy to see that our identity 
(\ref{idenity_distribution_01}) with $(\alpha,\rho)=(2,1/2)$, coincides with the result in (\ref{AP}) when $\theta=0$ (up to scaling).

Theorem \ref{thm_Mellin} also gives the following corollary. 

\begin{corollary}\label{corollary2}
Assume that $(\alpha,\rho_1)$ and $(\alpha,\rho_2)$ belong to ${\mathcal{A}}$. Then
\begin{equation}\label{identity_ratio}
\frac{X_{\alpha,\rho_1}}{X_{\alpha,\rho_2}} \stackrel{d}{=} \frac{X_{\alpha,\rho_2}}{X_{\alpha,\rho_1}}
\;\;\; \textnormal{ and } \;\;\;
\frac{\hat X_{\alpha,\rho_1}}{\hat X_{\alpha,\rho_2}} \stackrel{d}{=} \frac{\hat X_{\alpha,\rho_2}}{\hat X_{\alpha,\rho_1}},
\end{equation}
where all random variables are assumed to be independent, and in the second identity above we also assume that $\rho_1$ and $\rho_2$ are nonzero. 
\end{corollary}

The next theorem is our second main result, where we establish a convergent series representation for $\psi_{\alpha,\rho}(x)$ (we remind
the reader that  $\psi_{\alpha,\rho}(x)$ denotes the density of a free stable distribution $\nu_{\alpha,\rho}$). 
\begin{theorem}\label{theorem_series_expansions}
Assume that $\alpha \in (0,1)$ and $\rho \in [0,1]$ and denote $x^*:=\alpha(1-\alpha)^{{1}/{\alpha}-1}$.  Then
\begin{align}\label{series_x_greater_than_A}
\psi_{\alpha,\rho}(x)&=\frac{1}{\pi} \sum\limits_{n\ge 1} (-1)^{n-1}\frac{\Gamma(1+\alpha n)}{n!\Gamma(2+(\alpha-1)n)} 
\sin(n\alpha \rho \pi) x^{-\alpha n-1}, \;\;\; x \ge x^*,\\
\label{series_x_smaller_than_A}
\psi_{\alpha,\rho}(x)&=\frac{1}{\pi} \sum\limits_{n\ge 1} (-1)^{n-1}\frac{\Gamma\left(1+{n}/{\alpha}\right)}
{n!\Gamma\left(2+\left({1}/{\alpha}-1\right)n\right)} 
\sin(n\rho \pi) x^{n-1}, \;\;\; 0\le x \le x^*.
\end{align}
The corresponding series expansions for $\psi_{\alpha,\rho}(x)$ when $\alpha \in (1,2]$ and $\rho \in [1-1/\alpha,1/\alpha]$
can be obtained from \eqref{duality_identity1}, \eqref{series_x_greater_than_A} and \eqref{series_x_smaller_than_A}.
The expressions for $x<0$ follow from  $\psi_{\alpha,\rho}(x)=\psi_{\alpha,1-\rho}(-x)$. 
\end{theorem}

The series expansions \eqref{series_x_greater_than_A} and \eqref{series_x_smaller_than_A} have direct analogues in the case of classical stable distributions. The following result can be found in \cite[Theorems 2.4.2 and 2.5.1]{Zolotarev1986}. Let $p_{\alpha,\rho}(x)$ denote the density of a random variable $Y_{\alpha,\rho}$. Then 
for $\alpha\in(0,1)$ and $\rho\in[0,1]$ we have a convergent series representation
\begin{equation}\label{series_x_classical 1} 
p_{\alpha,\rho}(x)=\frac{1}{\pi} \sum\limits_{n\ge 1} (-1)^{n-1}\frac{\Gamma(1+\alpha n)}{n!} 
\sin(n\alpha \rho \pi) x^{-\alpha n-1}, \;\;\; x >0, 
\end{equation}
and an asymptotic expansion 
\begin{equation}\label{series_x_classical 2} 
p_{\alpha,\rho}(x)\sim\frac{1}{\pi} \sum\limits_{n\ge 1} (-1)^{n-1}\frac{\Gamma\left(1+{n}/{\alpha}\right)}
{n!}\sin(n\rho \pi) x^{n-1}, \;\;\; x \to 0^+.
\end{equation}
When $\alpha \in (1,2)$ the roles of the asymptotic and convergent series in \eqref{series_x_classical 1}  and 
\eqref{series_x_classical 2} are interchanged. 
We also note that the densities $\psi_{\alpha,\rho}$ and $p_{\alpha,\rho}$ can be represented as Fox $H$-functions \cite{Fox_H}, since their  Mellin transforms are ratios of products of Gamma functions.

Our third main result is the series expansion for the characteristic function of the free stable distribution, which we denote by
\begin{equation}\label{def_Fourier_free_stable}
{\mathfrak{f}}_{\alpha,\rho}(z):=\int_{\r} e^{\i zx } \nu_{\alpha,\rho}(\d x).
\end{equation}

\begin{theorem}\label{thm_f_alpha_rho_series}
Assume that $(\alpha,\rho) \in {\mathcal A}$.  Then for $z>0$
\begin{equation}\label{f_alpha_rho_series}
{\mathfrak{f}}_{\alpha,\rho}(z)=\sum\limits_{n\ge 0} (-1)^n \frac{e^{\pi \i \alpha ( {1}/{2} - \rho )n}}{n! \Gamma(2+(\alpha-1)n)} z^{\alpha n}.
\end{equation}
\end{theorem}

\vspace{0.25cm}
Note that for $z \in \r$ the value of ${\mathfrak{f}}_{\alpha,\rho}(z)$ is just the conjugate of ${\mathfrak{f}}_{\alpha,\rho}(-z)$, thus the
 series representation for ${\mathfrak{f}}_{\alpha,\rho}(z)$ for negative $z$ follows at once from \eqref{f_alpha_rho_series}. In light of our previous results, it should not be surprising that the infinite series in \eqref{f_alpha_rho_series} has its counterpart in the case of the classical stable distributions. Indeed, the characteristic function of the classical stable distribution $\mu_{\alpha,\rho}$ is given by
$$
{\mathfrak{g}}_{\alpha,\rho}(z)=\sum\limits_{n\ge 0} (-1)^n \frac{e^{\pi \i \alpha ( {1}/{2} - \rho )n}}{n!} z^{\alpha n}, \;\;\;\;\; z>0.
$$
The above result follows at once from \eqref{def_g_alpha_rho}.

\section{Proofs}

We use notation $\C^+:=\{z \in \C \; : \; \im(z)>0\}$ for the upper half-plane, and similarly
$\C^-:=\{z \in \C \; : \; \im(z)<0\}$ for the lower half-plane. 
For $s\in \C$ lying on the vertical line $\re(s)=0$ we define
\begin{equation}\label{def_Mellin_transform}
M_{\alpha,\rho}(s):=\e\left[(X_{\alpha,\rho})^s {\bf 1}_{\{X_{\alpha,\rho}>0\}}\right]=\int\limits_{0}^{\infty} x^{s} \psi_{\alpha,\rho}(x) \d x. 
\end{equation}

\vspace{0.25cm}
\noindent
{\bf Proof of Theorem \ref{thm_Mellin}:} 
Our first goal is to establish the identity \eqref{eq_Mellin_1} for $s\in (-1,0)$.  Assume that 
$(\alpha, \rho) \in {\mathcal A}$. 
Let $\gamma_{\alpha,\rho}\subset \bar \C^-$ be the the curve given in polar coordinates 
$$
\gamma_{\alpha,\rho}:=\{r(\theta)e^{-\i \theta} \; : \; 0<\theta < \pi\} \cup \{0\},
$$
 where
$$
r(\theta):=\Big(\frac{\sin(\theta)}{\sin[(1-\alpha\rho)\pi+(\alpha-1)\theta)]}\Big)^{1/\alpha}.
$$
As explained in Lemma A1.1 in \cite{BPB1999}, $\gamma_{\alpha,\rho}$ is a closed simple curve lying in the lower half-plane. We 
consider this curve as the boundary of a Jordan domain, denoted by $\Omega$. 
As was shown in Lemma A1.1 in \cite{BPB1999}, the function $G_{\alpha,\rho}(z)$ is a one-to-one conformal transformation from $\C^+$ onto $\Omega$, which extends continuously to the boundary. Thus $G_{\alpha,\rho}$ gives a homeomorphism of 
$\C^+ \cup \r \cup \{\infty\}$ with $\overline{\Omega}$. 
Let us define 
$$
z^*:=\exp(-\pi \i \rho).
$$
 One can check that $1/z^*+\phi_{\alpha,\rho}(1/z^*)=0$. This shows that 
$G_{\alpha,\rho}(z)$ (being the inverse of $1/z+\phi_{\alpha,\rho}(1/z)$) maps the positive real line $(0,\infty)$ onto the curve
$$
\gamma^+:=\{r(\theta)e^{-\i \theta} \; : \; 0\le \theta \le \pi\rho\}.
$$
Since $G_{\alpha,\rho}(0)=z^*$ and $G_{\alpha,\rho}(+\infty)=0$, the curve $\gamma^+$ is traversed in the direction from $z^*$ to $0$. 

The fact that $G_{\alpha,\rho}(z)$ is the inverse of $1/z+\phi_{\alpha,\rho}(1/z)$ can be expressed in the form
\begin{equation}\label{eqn_w_x}
w=G_{\alpha,\rho}(x) \Longleftrightarrow  x=w^{-1}-e^{\pi \i \alpha\rho} w^{\alpha-1}. 
\end{equation}
We know that $G_{\alpha,\rho}(x) \to 0$ as $x \to \infty$, and formula \eqref{eqn_w_x} shows that
\begin{equation}\label{eq_asymptotics_G}
G_{\alpha,\rho}(x)=w=\left( 1 - e^{\pi \i \alpha\rho} w^{\alpha} \right)/x=O(1/x), \;\;\; x\to \infty. 
\end{equation}
At the same time, since the function $1/z+\phi_{\alpha,\rho}(1/z)$ is analytic and has a non-vanishing derivative at $z=z^*$, 
its inverse function
$G_{\alpha,\rho}(z)$ is analytic in the neighborhood of zero. 
By the inversion formula for the Cauchy transform we have 
$$
\psi_{\alpha,\rho}(x)=-\frac{1}{\pi} \im( G_{\alpha,\rho}(x) ), \;\;\; x>0,
$$
therefore for $s\in (-1,0)$ 
\begin{equation}\label{formula_Ms_integral}
M_{\alpha,\rho}(s)=-\frac{1}{\pi} \im \left[ \int_{(0,\infty)} x^{s} G_{\alpha,\rho}(x) \d x   \right].
\end{equation}
The above integral exists for all $s\in (-1,0)$ due to \eqref{eq_asymptotics_G} and the fact that $G_{\alpha,\rho}(x)$ is 
analytic in the neighborhood of $x=0$. 

The main step of the proof is to perform a change of variables  $w=G_{\alpha,\rho}(x)$ in the above integral. 
From \eqref{eqn_w_x} we find
$$
\frac{\d x}{\d w}=-w^{-2}-(\alpha-1) e^{\pi \i \alpha\rho} w^{\alpha-2}.
$$  
Using the above result and the fact that $G_{\alpha,\rho}$ maps $(0,\infty)$ onto $\gamma^+$, we rewrite the integral in \eqref{formula_Ms_integral} as follows: 
\begin{align}\label{computing_M_s_eq1}
&\int_{(0,\infty)} x^{s} G_{\alpha,\rho}(x) \d x =\int_{\gamma^+} 
\left(w^{-1}-e^{\pi \i \alpha\rho} w^{\alpha-1}\right)^s w \left(-w^{-2}-(\alpha-1) e^{\pi \i \alpha\rho} w^{\alpha-2}\right) \d w\\ \nonumber
&=-\int_{\gamma^+} \left(1-\left(\frac{w}{z^*}\right)^{\alpha}\right)^s w^{-s-1} \d w
-(\alpha-1)e^{\pi \i \alpha\rho} \int_{\gamma^+} \left(1-\left(\frac{w}{z^*}\right)^{\alpha}\right)^s w^{-s+\alpha-1} \d w.
\end{align}
Note that both integrals in the right-hand side of \eqref{computing_M_s_eq1} are finite for $s\in (-1, 0)$. 
Let us compute the first integral in the right-hand side of  \eqref{computing_M_s_eq1}: we make a substitution
$w=uz^*$ and obtain
\begin{equation*}
I_1:=\int_{\gamma^+} \left(1-\left(\frac{w}{z^*}\right)^{\alpha}\right)^s w^{-s-1} \d w=
e^{\pi \i \rho s} \int_{L} (1-u^{\alpha})^s u^{-s-1} \d u,
\end{equation*}
where $L:= \gamma^+/z^*$, or, in other words, $L$ is the curve $\gamma^+$ rotated counterclockwise by the angle  
$\pi \rho$. It is clear that the curve $L$ connects points $0$ and $1$, it is traversed in the direction $1 \mapsto 0$, and 
$L\subset \C^+$ (except for the two endpoints). The 
function $w \mapsto f(w)=(1-w^{\alpha})^s w^{-s-1}$ extends to $\C^+$ analytically, and therefore we can deform the contour of integration 
$L$ 
into the interval $[0,1]$, and we finally obtain
\begin{align}\label{computing_I1}
\nonumber
I_1=e^{\pi \i \rho s} \int_{1}^0 (1-u^{\alpha})^s u^{-s-1} \d u&=
- \frac{1}{\alpha} e^{\pi \i \rho s}
\int_{0}^1 (1-x)^s x^{-{s/\alpha}-1} \d x\\
&=- \frac{1}{\alpha} e^{\pi \i \rho s}
\frac{\Gamma(s+1)\Gamma(-{s/\alpha})}{\Gamma(1+s-{s/\alpha})}.
\end{align}
When computing the above integral, in the second step we made a substitution $u=x^{\frac{1}{\alpha}}$ and in the third step we have used the well-known integral representation for the beta function. 
We deal with the second integral in the right-hand side of 
\eqref{computing_M_s_eq1} in exactly the same way, and we find that for $s\in (-1,0)$ 
\begin{equation}\label{computing_I2}
I_2:=\int_{\gamma^+} \left(1-\left(\frac{w}{z^*}\right)^{\alpha}\right)^s w^{-s+\alpha-1} \d w
=- \frac{1}{\alpha}  e^{\pi \i \rho (s-\alpha)}
\frac{\Gamma(s+1)\Gamma(1-{s/\alpha})}{\Gamma(2+s-{s/\alpha})}.
\end{equation}
Combining formulas \eqref{formula_Ms_integral}, \eqref{computing_M_s_eq1}, \eqref{computing_I1} and 
\eqref{computing_I2} we obtain
\begin{align*}
M_{\alpha,\rho}(s)&=-\frac{1}{\pi} \im \bigg[ \frac{1}{\alpha} e^{\pi \i \rho s}
\frac{\Gamma(s+1)\Gamma(-{s/\alpha})}{\Gamma(1+s-{s/\alpha})}
\\
 & \qquad\qquad - (\alpha-1)e^{-\pi \i (1-\alpha \rho)} \frac{1}{\alpha}  e^{\pi \i \rho(s-\alpha)}
\frac{\Gamma(s+1)\Gamma(1-{s/\alpha})}{\Gamma(2+s-{s/\alpha})}
 \bigg]\\
 &=-\frac{1}{\pi\alpha} \sin\left(\pi\rho s \right) \bigg[
\frac{\Gamma(s+1)\Gamma(-{s/\alpha})}{\Gamma(1+s-{s/\alpha})}
+ (\alpha-1)
\frac{\Gamma(s+1)\Gamma(1-{s/\alpha})}{\Gamma(2+s-{s/\alpha})}
 \bigg]\\
 &=-\frac{1}{\pi\alpha} \sin\left(\pi\rho s \right)
 \frac{\Gamma(s)\Gamma(1-{s/\alpha})}{\Gamma(2+s-{s/\alpha})}
 \left[ -\alpha(1+s-{s/\alpha})+(\alpha-1)s \right]\\
 &=\frac{1}{\pi}\sin\left(\pi\rho s \right)
 \frac{\Gamma(s)\Gamma(1-{s/\alpha})}{\Gamma(2+s-{s/\alpha})}.
\end{align*}
This ends the proof of \eqref{eq_Mellin_1} for $s \in (-1,0)$. The extension of the result for $s\in (-1,\alpha)$ follows by analytic continuation. 
\qed

The above proof of Theorem \ref{thm_Mellin} is similar in spirit to the proof of Lemma 10 in \cite{HM2013}. 
One can also give an alternative proof of Theorem \ref{thm_Mellin}, based on the following four steps.
\begin{itemize}
\item[(i)] The Mellin transform $M_{\alpha,1}(z)$ for $\alpha\in(0,1)$ is known due to Theorem 3 in \cite{HM2013}. 
\item[(ii)] Assume that $\alpha \in (0,1)$ and $\rho \in [0,1]$. One can prove that 
\begin{equation}\label{factorization stable Cauchy}
X_{\alpha,\rho}\stackrel{d}{=} X_{\alpha,1} \times K_{\rho}, 
\end{equation}
where the random variable $K_\rho$ is independent of $X_{\alpha,\rho}$ and follows the Cauchy distribution 
\begin{equation}\label{cauchy density}
\p(K_{\rho} \in \d x)=\frac{1}{\pi} \times \frac{\sin(\pi \rho)}{(x+\cos(\pi\rho))^2+\sin^2(\pi\rho)} \d x. 
\end{equation}
The factorization (\ref{factorization stable Cauchy}) can be established as follows. Let $G_{X}(z)$ denote the Cauchy transform $\e[1/(z-X)]$ of a random variable $X$. First note the fact $G_{K_\rho}(z)=1/(z+e^{\pi \i \rho})$ which can be proved by the residue theorem (this formula and the Cauchy transform inversion formula give the density (\ref{cauchy density})).
 Then for $z \in \C^+$ we have 
\begin{equation}\label{Cauchy2}
\begin{split}
G_{X_{\alpha,1} \times K_\rho}(z)
&=\e\!\left[\frac{1}{z-X_{\alpha,1} \times K_\rho}\right]  
= \e\!\left[\frac{1}{X_{\alpha,1}}\times \frac{1}{z/X_{\alpha,1} - K_\rho}\right]  \\
&=\e\!\left[\frac{1}{X_{\alpha,1}}\times \frac{1}{z/X_{\alpha,1} +e^{\pi \i \rho}}\right] 
= \e\!\left[\frac{ -e^{-\pi \i \rho}}{-z e^{-\pi\i \rho}- X_{\alpha,1}}\right] \\
&=  -e^{-\pi \i \rho}G_{X_{\alpha,1}}( -e^{-\pi \i \rho} z). 
\end{split}
\end{equation}
Recall the fact that for any random variable $X$ and any $a>0$, there exists $b>0$ such that $G_X$ has the right compositional inverse $G_X^{-1}$ defined in $$\Delta_{a,b}=\{z\in \C^- \mid \im z \in(-b,0), a |\re z| \leq -\im z \},$$ such that 
\begin{equation}\label{Voi}
G_X^{-1}(z)= 1/z +\phi_X(1/z),\qquad z\in\Delta_{a,b}; 
\end{equation}
 see \cite{BV1993}.  From (\ref{Cauchy2})  and (\ref{Voi}) one has
 $$
\phi_{X_{\alpha,1} \times K_\rho}(z)= -e^{\i \pi \rho}\phi_{X_{\alpha,1}}(-e^{-\i \pi \rho}z) = \phi_{X_{\alpha,\rho}}(z) 
$$
in a common domain. This shows the factorization (\ref{factorization stable Cauchy}). 

\item[(iii)] The Mellin transform of the Cauchy distribution is given by  
$$
\e[(K_\rho)^s {\bf 1}_{\{K_\rho>0\}}] = \frac{\sin \pi \rho s}{\sin \pi s},\qquad -1<s<1.
$$

\item[(iv)] The explicit expression for $M_{\alpha,\rho}(z)$ for $\alpha\in(0,1)$ and $\rho\in[0,1]$ follows from 
the factorization  $X_{\alpha,\rho}\stackrel{d}{=} X_{\alpha,1} \times K_{\rho}$. We then use the duality identity \eqref{duality_identity1} to obtain $M_{\alpha,\rho}(z)$ for $\alpha\in(1,2]$ and $\rho\in[1-{1/\alpha},{1/\alpha}]$. 
\end{itemize}

\vspace{0.25cm}
\noindent
{\bf Proof of Corollary \ref{corollary_duality}:} 
Formula \eqref{eq_Mellin_1} implies that for all $(\alpha,\rho)\in {\mathcal A}$ and $-1<s<\alpha^{-1}$ we have the identity
$M_{\alpha,\rho}(-\alpha s)=(1/\alpha) \times M_{1/\alpha,\alpha\rho}(s)$, which is equivalent to
\eqref{duality_identity1} and \eqref{duality_identity2}. 
\qed

\vspace{0.25cm}
\noindent
{\bf Proof of Corollary \ref{corollary1}:} 
Assume that $s\in \C$ and $\re(s)=0$. It is easy to see that
$$
\e\left[Z^{(1-{1/\alpha})s}\right]=\Gamma\left(2+\left(1-{1/\alpha}\right)s\right).
$$
Using the above identity and formulas \eqref{eq_Mellin_1} and \eqref{eq_Mellin_1_classical}  we obtain
$$
\e\left[\left(Y_{\alpha,\rho}\right)^s {\bf 1}_{\{Y_{\alpha,\rho}>0\}}\right] =
\e\left[Z^{(1-{1/\alpha})s} \right] \times \e\left[\left(X_{\alpha,\rho}\right)^s {\bf 1}_{\{X_{\alpha,\rho}>0\}}\right].
$$
The above result and the uniqueness of the Mellin transform imply that 
\begin{equation}\label{identity_distribution}
\p(Y_{\alpha,\rho} \in A)=\p (X_{\alpha,\rho} \times 
Z^{1-{1/\alpha}} \in A),
\end{equation}
 for all Borel sets $A \subset (0,\infty)$. 
The fact that \eqref{identity_distribution} also holds for all Borel sets $A \subset (-\infty,0)$ follows by using the symmetry condition
$-X_{\alpha,\rho} \stackrel{d}{=} X_{\alpha,1-\rho}$ and $-Y_{\alpha,\rho} \stackrel{d}{=} Y_{\alpha,1-\rho}$. This ends the proof of
\eqref{idenity_distribution_01}. 
\qed

\begin{lemma}\label{lemma_division}
Assume that $X, Y, Z$ and $W$ are independent random variables, which satisfy the following four conditions: (i) $X>0$ and $Y>0$ a.s.,  
(ii) $X \stackrel{d}{=} Y$, (iii) $X \times Z \stackrel{d}{=} Y \times W$ and (iv) there exists $\epsilon>0$ such that
$\e[|Z|^s]<\infty$ and $\e[|W|^s]<\infty$ for all $s\in (-\epsilon,\epsilon)$. 
Then $Z \stackrel{d}{=} W$.
\end{lemma}
\begin{proof}
Recall that we denote by $\hat \xi$ the cutoff of a random variable $\xi$. Assume that $\p(Z>0)>0$. 
Condition (iii) implies that $\p(Z>0)=\p(W>0)$ and $X \times \hat Z \stackrel{d}{=} Y \times \hat W$.
For $t\in \r$ we denote $f(t)=\e[X^{\i t}]=\e[Y^{\i t}]$ and obtain
$$
f(t) \times \e[\hat Z^{\i t}] = \e[(X\hat Z)^{\i t}] = \e[(Y\hat W)^{\i t}]=f(t) \times \e[\hat W^{\i t}]. 
$$
Since $f(t)$ is continuous (as a characteristic function of $\ln(X)$) and $f(0)=1$ we conclude that for some $\delta$ small enough, the function $f(t)$ is non-zero for $t\in (-\delta,\delta)$. Therefore, we can divide both sides of the 
above identity by $f(t)$ and conclude that $\e[\hat Z^{\i t}]=\e[\hat W^{\i t}]$ for $t\in (-\delta,\delta)$. Condition (iv) implies that both functions
 $\e[\hat Z^{\i t}]$ and $\e[\hat W^{\i t}]$ are analytic in the strip $\im(t) \in (-\epsilon,\epsilon)$, and since they are equal on the set 
 $t\in (-\delta,\delta)$, they must be equal for all $t$ in the strip $\im(t) \in (-\epsilon,\epsilon)$, which implies $\hat Z \stackrel{d}{=}  \hat W$. 
 
 In the case if $\p(Z<0)>0$ we can apply the same argument as above and show that the cutoff of $-Z$ has the same distribution as the cutoff of $-W$. 
\end{proof}

\noindent
{\bf Proof of Corollary \ref{corollary2}:} This result follows at once from Corollary \ref{corollary1}, Lemma \ref{lemma_division}, and the following fact: the identity \eqref{identity_ratio} is true for classical stable random variables $Y_{\alpha,\rho}$ and 
their cutoffs $\hat Y_{\alpha,\rho}$ (see formulas (3.2.3) and (3.2.5) in \cite{Zolotarev1986}). \qed

\vspace{0.25cm}
\noindent
{\bf Proof of Theorem \ref{theorem_series_expansions}:} 
First of all, let us check that the series in \eqref{series_x_greater_than_A} and \eqref{series_x_smaller_than_A}  converge for $x=x^*$.
We recall Stirling's asymptotic formula for the gamma function: For every $\epsilon>0$ 
\begin{equation}\label{Stirling_formula}
\ln (\Gamma(z))=\left(z-\tfrac{1}{2}\right) \ln(z)-z+\tfrac{1}{2}\ln(2\pi)+O_{\epsilon}(z^{-1}), 
\end{equation}
as $|z| \to \infty$, uniformly in the sector $|\arg(z)|<\pi-\epsilon$
Using \eqref{Stirling_formula} and the reflection formula for the gamma function we check that
$$
 \frac{\Gamma(1+\alpha n)}{n!\Gamma(2+(\alpha-1)n)} 
=\frac{1}{\pi}  \frac{\Gamma(1+\alpha n)}{n!} \sin(\pi (\alpha-1) n)  \Gamma(-1+(1-\alpha)n) 
=O\left(n^{-{3/2}} (x^*)^{\alpha n}\right),  
$$
as $n\to +\infty$, which shows that the series in \eqref{series_x_greater_than_A} converges for $x=x^*$ (therefore, it converges uniformly on $x\in [x^*, \infty)$). 
In the same way we check that
$$
\frac{\Gamma\left(1+{n/\alpha}\right)}
{n!\Gamma\left(2+\left({1/\alpha}-1\right)n\right)}   = O\left(n^{-{3/2}} (x^*)^{-n}\right), \;\;\; n\to +\infty,
$$
therefore the series in \eqref{series_x_smaller_than_A} converges for $x=x^*$ (and it converges uniformly for $x\in [-x^*, x^*]$). 

Let us prove identity \eqref{series_x_greater_than_A}.  We recall that $M_{\alpha,\rho}(s)$ is defined by \eqref{def_Mellin_transform}, and 
we have already established that it is equal to the function in the right-hand side 
of \eqref{eq_Mellin_1}. It is clear that $M_{\alpha,\rho}(s)$
has simple poles at points 
$$
s_n:=\alpha n, \; n\ge 1 \;\;\; \textnormal{ and } \;\;\; \hat s_m:=-m, \; m\ge 1. 
$$
The poles at $s_n$ \{respectively, $\hat s_m$\} come from the factor $\Gamma(1-s/\alpha)$ \{respectively, $\Gamma(s)$\} in \eqref{eq_Mellin_1}.
The corresponding residues are given by 
\begin{align}\label{residues_of_M}
&{\textnormal{Res}}(M_{\alpha,\rho}(s) : s=s_n)=\frac{1}{\pi} (-1)^{n}\frac{\Gamma(1+\alpha n)}{n!\Gamma(2+(\alpha-1)n)} 
\sin(n\alpha \rho \pi), \\ \label{residues_of_M2}
&{\textnormal{Res}}(M_{\alpha,\rho}(s)  :  s=\hat s_m)=\frac{1}{\pi} (-1)^{m-1}\frac{\Gamma\left(1+{m/\alpha}\right)}
{m!\Gamma\left(2+\left({1/\alpha}-1\right)m\right)} 
\sin(m\rho \pi).
\end{align}
Using the reflection formula for the Gamma function we check that $M_{\alpha,\rho}(s) \equiv f_1(s) \times f_2(s)$, where
we have defined
$$
f_1(s):=-\frac{1}{\pi} \frac{\sin(\pi \rho s) \sin\left(\pi \left({1/\alpha}-1\right)s\right)}
{\sin\left({\pi s/\alpha}\right)}, \;\;\;
f_2(s):=\frac{\Gamma(s) \Gamma\left(-1-s+{s/\alpha}\right)}{\Gamma\left({s/\alpha}\right)}.
$$

Let us denote $B_r(w):=\{z\in\C \mid |z-w|<r\}$ and $\delta:=\alpha/4$. The function $f_1(s)$ has poles at points
$s=n\alpha$, $n\in {\mathbb Z}$, and it satisfies $f_1(s)=O(\exp(-\pi (1-\rho)|\im(s)|))$ as $\im(s) \to \infty$. This implies
that for some $C_1>0$ we have
\begin{equation}\label{f_1_asymptotics}
|f_1(s)|\leq C_1,\qquad s \notin \bigcup_{n=0}^\infty B_{\delta}(n \alpha),\quad \re(s)\geq0. 
\end{equation}
From Stirling's formula \eqref{Stirling_formula} we find that there exist $C_2>0$ and $C_3>0$ such that 
\begin{equation}\label{f_2_asymptotics}
|f_2(s)|<C_2 |s|^{-{3/2}}(x^*)^{\re(s)}, \;\;\; |s|>C_3, ~\re(s)>0. 
\end{equation}
Formulas \eqref{f_1_asymptotics} and \eqref{f_2_asymptotics} show that $M_{\alpha,\rho}(s)$ is absolutely integrable on the vertical line $\i \r$, therefore we may use the Mellin transform inversion formula
$$
\psi_{\alpha,\rho}(x)=\frac{x^{-1}}{2\pi \i} \int_{\i \r} M_{\alpha,\rho}(s) x^{-s} \d s.
$$
Let us define $b_k:=\alpha(2 k+1)/2$.
Shifting the contour of integration $\i \r \mapsto b_k + \i \r$ we obtain
\begin{equation}\label{shifting_contour}
\psi_{\alpha,\rho}(x)=-\sum\limits_{n=1}^{k} {\textnormal{Res}}(M_{\alpha,\rho}(s) : s=s_n) \times x^{-s_n-1}
+\frac{x^{-1}}{2\pi \i} \int_{b_k+\i \r} M_{\alpha,\rho}(s) x^{-s} \d s.
\end{equation}
Let us denote the integral in the right-hand side of \eqref{shifting_contour} by $I_k(x)$. 
Changing the variable of integration $s=b_k+\i u$ we find 
\begin{equation}\label{estimate_I_n1}
|I_k(x)|= 
\left| \i \int_{\r} M_{\alpha,\rho}(b_k + \i u) x^{-b_k-\i u} \d u \right|
\le x^{-b_k} \int_{\r} |f_1(b_k+\i u)| \times |f_2(b_k+\i u)| \d u.
\end{equation}
 Note that due to our choice of $b_k$, the vertical line $b_k + \i \r$ does not intersect the collection of circles
$\bigcup_{n=0}^\infty B_{\delta}(n \alpha)$, thus the estimate \eqref{f_1_asymptotics} holds true for all $s \in b_k + \i \r$. 
Estimates  \eqref{f_1_asymptotics}, \eqref{f_2_asymptotics} and  \eqref{estimate_I_n1} show that for all $b_k > C_3$  we have
$$
|I_k(x)|\le C_1C_2  \left(\frac{x^*}{x}\right)^{b_k}   \int_{\r}(b_k^2+u^2)^{-{3/4}} \d u,
$$
which implies that $I_k(x)\to 0$ as $k \to +\infty$, provided that  $x \ge x^*$. Combining this statement with 
formulas \eqref{residues_of_M} and \eqref{shifting_contour} gives us the desired series expansion \eqref{series_x_greater_than_A}. 

The proof of \eqref{series_x_smaller_than_A} can be obtained in exactly the same way, except that now we should shift the 
contour of integration in the opposite direction, while taking into account the contribution from the simple poles at points $\hat s_m$. The details of the proof are left to the reader. 
\qed

\begin{lemma}\label{lemma_eqn_fourier_mellin}
Assume that $(\alpha,\rho) \in {\mathcal A}$. For $s>0$ we have
\begin{equation}\label{eqn_fourier_mellin}
\int_0^{\infty} {\mathfrak{f}}_{\alpha,\rho}(z) z^{s-1}\d z=\frac{1}{\alpha} e^{\pi \i s \left( \rho-{1/2} \right)} 
\frac{\Gamma\left({s/\alpha}\right)}{\Gamma\left(2-s+{s/\alpha}\right)}. 
\end{equation}
\end{lemma}
\begin{proof}
We use the following result (see Lemma 1 in \cite{Simon_Profeta}): {\it Let $s \in (0,1)$ and assume that $X$ is a
random variable such that $\e[|X|^{-s}]<\infty$. Then}
\begin{align*}
\int_0^{\infty} z^{s-1} \e[\cos(zX)] \d z&=\Gamma(s) \cos\left({\pi s/2}\right) \e[|X|^{-s}], \\
\int_0^{\infty} z^{s-1} \e[\sin(zX)] \d z&=\Gamma(s) \sin\left(\pi s/2\right) \e[|X|^{-s}{\textnormal{sign}}(X)].
\end{align*}
Combining the above two identities we obtain the following result: for $s \in (0,1)$ 
\begin{align*}
\int_0^{\infty} {\mathfrak{f}}_{\alpha,\rho}(z) z^{s-1}\d y
&=
\Gamma(s) \cos\left({\pi s/2}\right) 
\left[M_{\alpha,\rho}(-s)+M_{\alpha,1-\rho}(-s)\right]\\
&+
\i \Gamma(s) \sin\left({\pi s/2}\right) 
\left[M_{\alpha,\rho}(-s)-M_{\alpha,1-\rho}(-s)\right],
\end{align*}
where $M_{\alpha,\rho}(s)$ is defined by \eqref{def_Mellin_transform}. The required result
\eqref{eqn_fourier_mellin} follows from the above identity, formula \eqref{eq_Mellin_1} and the following elementary trigonometric identity
\begin{align*}
&\cos\left({\pi s/2}\right) \left[\sin(\pi \rho s)+\sin(\pi (1-\rho)s)\right]+
\i \sin\left({\pi s/2}\right) \left[\sin(\pi \rho s)-\sin(\pi (1-\rho)s)\right]\\&
\qquad \qquad\qquad \qquad\qquad \qquad\qquad =
\sin(\pi s) e^{\pi \i s \left(\rho-1/2 \right)}.
\end{align*}
Extension of \eqref{eqn_fourier_mellin} for all $s>0$ is achieved by analytic continuation, since 
the right-hand side of \eqref{eqn_fourier_mellin} is  analytic in the half-plane $\re(s)>0$. 
\end{proof}

\vspace{0.25cm}
\noindent
{\bf Proof of Theorem \ref{thm_f_alpha_rho_series}:}
The proof uses the Mellin transform inversion technique, and it is very similar to the proof of Theorem \ref{theorem_series_expansions}. 
Therefore, we only sketch the main steps of the proof and we leave all the details to the reader.

Let us denote the function in the right-hand side of \eqref{eqn_fourier_mellin} by $m_{\alpha,\rho}(s)$. Assume that $\alpha \in (0,1)$
and $\rho \in (0,1)$. Using the 
reflection formula for the Gamma function we find that $m_{\alpha,\rho}(s)=f_1(s) \times f_2(s)$, where we have defined
$$
f_1(s):=e^{\pi \i s \left( \rho-1/2 \right)} \frac{\sin(({1/\alpha}-1)\pi s)}{\sin( {\pi s/\alpha})}, \;\;\;
f_2(s):=\frac{\Gamma(-1-({1/\alpha}-1)s)}{\Gamma(1-{s/\alpha})}. 
$$
Using Stirling's asymptotic formula \eqref{Stirling_formula} we check that the function $u\in \r \mapsto m_{\alpha,\rho}(\i u)$ 
converges to zero exponentially fast as $u \to \infty$, thus we can express ${\mathfrak{f}}_{\alpha,\rho}(x)$ as the inverse Mellin transform
\begin{equation}\label{Inv Mellin}
{\mathfrak{f}}_{\alpha,\rho}(z)=\frac{1}{2\pi \i} \int_{1+\i \r} 
m_{\alpha,\rho}(s) z^{-s} \d s. 
\end{equation}
We define $b_k:=\alpha (2k+1)/2$ and  shift the contour of integration in \eqref{Inv Mellin} $1+\i \r \mapsto -b_k + \i \r$. Taking into account the residues of the integrand at points $s=-\alpha n$ (coming from the factor $\Gamma(s/\alpha)$ in \eqref{eqn_fourier_mellin}),
we obtain
\begin{equation}\label{shift_contour_fourier}
{\mathfrak{f}}_{\alpha,\rho}(z)=\sum\limits_{0 \le n \le k} (-1)^n \frac{e^{\pi \i \alpha ( {1/2} - \rho )n}}{n! \Gamma(2+(\alpha-1)n)} z^{\alpha n}+\frac{1}{2\pi \i} \int_{-b_k+\i \r} 
m_{\alpha,\rho}(s) x^{-s} \d s. 
\end{equation}
It is easy to see that for some $C_1>0$ we have 
$$
|f_1(-b_k+\i u)|<C_1 \exp\left(- \pi (1-|\rho-1/2|) |u|\right),
$$
for all $k \ge 0$ and $u\in \r$.
Stirling's asymptotic formula \eqref{Stirling_formula} shows that there exist constants $C_2 \in \r$ and $C_3>0$ such that
\begin{equation*}
\ln(f_2(s))=s \ln(-s)+C_2s+O(1)
\end{equation*}
as $s \to \infty$, uniformly in the half-plane $\re(s)<-C_3$. The above two estimates can be used to show that the integral in the right-hand side of \eqref{shift_contour_fourier} converges to zero as $k\to +\infty$, which ends the proof of 
\eqref{f_alpha_rho_series} for $\alpha \in (0,1)$ and $\rho \in (0,1)$. The extension of the result in the case 
$\rho \in \{0,1\}$ follows by considering the limit of \eqref{f_alpha_rho_series} as $\rho \to 0^+$ or $\rho \to 1^-$. The proof in the case $\alpha \in (1,2]$ can be obtained along the same lines. 
\qed

\section*{Acknowledgements} The authors would like to thank two anonymous referees for their careful reading of the paper and for suggesting several improvements. The research of T. Hasebe was supported by Marie Curie Actions -- International Incoming Fellowships (Project 328112 ICNCP). The research of A. Kuznetsov was supported by the Natural Sciences and Engineering Research Council of Canada.





\end{document}